\documentclass[a5paper,11pt]{amsart}
\usepackage{amsmath, amsthm, amstext}
\usepackage{amssymb, array, amsfonts}
\usepackage[english]{babel}
\usepackage{bbding}
\usepackage{float}
\usepackage{graphics}
\usepackage{graphicx}
   \usepackage{enumerate}
\usepackage{epstopdf}
\numberwithin{equation}{section}

\def \NN{\mathrm N}

\def \er{\varepsilon}
\def \eps{\varepsilon}

\def \F {\mathfrak F}

\def \ph{\varphi}

\renewcommand{\l}{\left}
\renewcommand{\r}{\right}

\def \A{\mathcal{A}}
\def \B{\mathcal{B}}
\def \C{\mathbb{C}}

\def \GL{\mathrm{GL}}

\def \K{\mathcal{K}}
\def \M2{\mathrm{M}_2}

\def \O{\mathcal O}
\def \R{\mathbb{R}}

\def \Z{\mathbb{Z}}

\def\dr{\mathrm d}

\newcommand{\overbar}[1]{\mkern 1.5mu\overline{\mkern-1.5mu#1\mkern-1.5mu}\mkern 1.5mu}

\newcommand{\beq}{\begin{equation}}
\newcommand{\eeq}{\end{equation}}

\def\Ran{\operatorname{Ran}}

\def\diag{\operatorname{diag}}
\def\ind{\operatorname{ind}}

\def\OL{\operatorname{OL}}

\def\ran{\operatorname{ran}}

\def\supp{\operatorname{supp}}
\def\ess{\operatorname{ess}}

\def\diam{\operatorname{diam}}
\def\im{\operatorname{Im}}
\def\re{\operatorname{Re}}

\newcommand{\eqdef}{\stackrel{\rm def}{=\kern-3.6pt=}}

\theoremstyle{plain}
\newtheorem{theorem}{\bf Theorem}[section]
\newtheorem{lemma}[theorem]{\bf Lemma}

\newtheorem{cor}[theorem]{\bf Corollary}

\theoremstyle{definition}
\newtheorem{defi}[theorem]{\bf Definition}

\theoremstyle{remark}
\newtheorem{remark}[theorem]{\bf Remark}

\renewcommand{\le}{\leqslant}
\renewcommand{\ge}{\geqslant}
\newcommand{\dist}{\mathop{\mathrm{dist}}\nolimits}

\renewcommand{\qed}{\vrule height7pt width5pt depth0pt}

\begin{document}

\title[Distance to normal elements]{Distance to normal elements\\ in $C^*$-algebras of real rank zero}
\author[I. Kachkovskiy]{Ilya Kachkovskiy}
\address{Department of Mathematics,
University of California, Irvine,
Irvine, CA, 92697-3875,
United States of America}
\thanks{The first author was supported by King's Annual Fund and King's Overseas Research Studentships, King's College London, and partially by NSF Grant DMS-1101578.}
\email{ilya.kachkovskiy@uci.edu}
\author[Y. Safarov]{Yuri Safarov}
\address{
Department of Mathematics,
King's College London,
Strand, London WC2R 2LS,
United Kingdom }
\email{yuri.safarov@kcl.ac.uk}

\date{}



\thanks{The authors are grateful to N. Filonov for helpful remarks and suggestions. We should also like to thank the referees for their comments and constructive criticism.}

\begin{abstract}
We obtain an order sharp estimate for the distance from a given bounded operator $A$ on a Hilbert space to the set of normal operators in terms of $\|[A,A^*]\|$ and the distance to the set of invertible operators. A slightly modified estimate holds in a general $C^*$-algebra of real rank zero.
\end{abstract}

\maketitle

\section*{Introduction}

The problem of estimating the distance from a bounded operator $A$ to the set of normal operator in terms of $\|[A,A^*]\|$ dates back to Halmos \cite{Halmos}. The original question is as follows.

\begin{enumerate}
\item[{\bf(C)}] 
Is there a continuous function $F$ on $\R_+$ with $F(0)=0$ such that
for each pair of Hermitian matrices $X,Y$ with $\|X\|+\|Y\|\le1$ there exists a pair of commuting Hermitian matrices $X',Y'$ satisfying the estimate 
$\|X-X'\|+\|Y-Y'\|\le F\bigl(\|[X,Y]\|\bigr)$?
\end{enumerate} 
Introducing $A=X+iY$, one can reformulate {\bf(C)} in terms of one operator:
\begin{enumerate}
\item[({\bf C$'$})]
is there a continuous function $F$ on $\R_+$ with $F(0)=0$ such that
the distance from a matrix $A$ with $\|A\|\le1$ to the set of normal matrices does not exceed $F\bigl(\|[A,A^*]\|\bigr)$?
\end{enumerate}

Clearly, the answer is positive if we allow $F$ to depend on the dimension. A survey of dimension-dependent results can be found in \cite[Chapter I]{DSz}.

The problem of the existence of a dimension-independent function $F$ is much more challenging. It was open until 1995, when Huaxin Lin found deep $C^*$-algebraic arguments showing that such a function $F$ does exist \cite{L}. Later, Friis and R{\o}rdam gave a shorter proof of Lin's theorem \cite{FR}. Note that, without additional conditions on $X$ and $Y$, Lin's theorem does not hold for non-Hermitian matrices or self-adjoint operators acting on a Hilbert space (see, for instance, \cite{Ch,D,DSz}).

The proofs in \cite{FR,L} are non-constructive and do not give any information about the function $F$. To the best of our knowledge, the only quantitative result in this direction is due to Hastings, who showed in \cite{H_orig} that the distance from $A$ to the set of normal matrices is estimated by $C_\eps\,\|[A,A^*]\|^{1/5-\eps}$ for all finite matrices $A$ with $\|A\|\le1$, where $C_\eps$ is  a constant depending on $\eps>0$.  
For homogeneity reasons, the function $F$ in {\bf(C)} or {\bf(C$'$)} cannot decay faster than $t^{1/2}$ as $t\to0$. If we drop the condition $\|A\|\le1$, it also cannot grow slower than $t^{1/2}$ as $t\to\infty$. 
In \cite{DSz}, Davidson and Szarek conjectured that for finite matrices one can indeed take $F(t)=C\,t^{1/2}$.

A closely related result is the famous Brown--Douglas--Fillmore (BDF) theorem  \cite{BDF}. Recall that an operator $A$ on a Hilbert space is said to be essentially normal if its self-commutator $[A,A^*]$ is compact. The BDF theorem states that a bounded essentially normal operator $A$ on a separable Hilbert space is a compact perturbation of a normal operator if and only if $A$ has trivial index function (that is, $\ind(A-zI)=0$ whenever the operator $A-zI$ is Fredholm).  In \cite{FR2}, the authors gave a simple proof of this theorem, which essentially repeats their proof of Lin's theorem in \cite{FR}. 

In \cite{BD}, Berg and Davidson obtained a quantitative version of the BDF theorem. They proved that for each bounded closed set $\Omega\subset\C$ there exists a continuous function $F_\Omega$ with $F_\Omega(0)=0$ such that  the following is true. If $A$ is an essentially normal operator with trivial index function, $\|[A,A^*]\|^{1/2}\le\eps$ and, in addition, $\|(A-\lambda I)^{-1}\|<\left(\dist(\lambda,\Omega)-\eps\right)^{-1}$ whenever $\dist(\lambda,\Omega)>\eps$ then there is a normal operator $N_\eps$ with spectrum in $\Omega$ such that $A-N_\eps$ is compact and $\|A-N_\eps\|\le F_\Omega(\eps)$. Note that proofs in \cite{BD} could be simplified by applying Lin's theorem, which was not known at that time. Instead, the authors used an absorption result of Davidson \cite{D}.

The main result of this paper is Theorem \ref{t:main}, which gives explicit bounds for the distance to the set of normal elements in an abstract $C^*$-algebra. It refines and extends all the results mentioned above. In particular, Theorem \ref{t:main} implies
\begin{enumerate}
\item[--] the estimate \eqref{main-matrices} showing that the conjecture from \cite{DSz} is true,
\item[--] Berg and Davidson's theorem with $F_\Omega(\eps)=C\eps$, where $C$ is a constant independent of $\Omega$ (see Remark \ref{bdf_rem}), and
\item[--] a quantitative version of the BDF theorem, which holds for operators with non-compact self-commutators (see Corollary \ref{c:bdf}).
\end{enumerate}

\section{Notation and results}
\label{s:main}

\subsection{Main theorem} 
Let $\A$ be a unital $C^*$-algebra. Recall that $\A$ is said to have {\it real rank zero} if any 
its self-adjoint element can be approximated by self-adjoint elements with finite spectra.
Further on
\begin{itemize}
\item 
$\GL(\A)$ denotes the group of invertible elements of $\A$;
\item
$\GL_0(\A)$ is 
the connected component of $\GL(\A)$ containing the identity;
\item 
$\NN(\A)$ denotes the set of normal elements of $\A$;
\item 
$\NN_f(\A)$ is the set of normal elements with finite spectra;
\item
$d_1(A):=\sup\limits_{\lambda\in \C}\dist(A-\lambda I,\GL_0(\A)).$
\end{itemize}

\begin{theorem}
\label{t:main}
For any unital $C^*$-algebra $\A$ of real rank zero and all $A\in \A$ 
\begin{align}\label{main-estimate1}
\dist\left(A,\NN_f({\A})\right) &\le C\left(\|[A,A^*]\|^{1/2}+d_1(A)\right)\,,\\
\label{main-estimate2}
\dist\left(A,\NN_f({\A})\right) &\ge
\max\left\{(5\|A\|)^{-1}\|[A,A^*]\|\,,\,d_1(A)\right\},
\end{align}
where $C$ is a constant independent of $\A$ and $A$.
\end{theorem}

\begin{remark}
\label{n-f}
All von Neumann algebras, including the algebra of bounded operators $\B(H)$ on a Hilbert space $H$, have real rank zero. If $\A$ is a von Neumann algebra then $\NN(\A)\subset\overline{\NN_f(\A)}$ and $\GL_0(\A)=\GL(\A)$, so that we can drop the subscripts $f$ in \eqref{main-estimate1}, \eqref{main-estimate2} and $0$ in the definition of $d_1(A)$. 
\end{remark}

\begin{remark}
\label{gl-0}
If the complement of the spectrum of $A$ is connected and dense in $\C$ then 
$d_1(A)=0$. Indeed, in this case $A-\lambda I$ can be approximated by invertible elements of the form $A-\mu I$, and each invertible $A-\mu I$ belongs to $\GL_0(\A)$ because $-\mu^{-1}(A-\mu I)\to I$ as $\mu\to\infty$. In particular, $d_1(A)=0$ for finite matrices and, more generally, compact operators $A$. If $H$ is separable and $\A=\B(H)$ then $d_1(A)=0$ if and only if $A$ has trivial index function (see, for instance, \cite[Section 3.2]{FS} and references therein).
In the general case, $d_1(A)$ can be estimated in terms of the so-called modulus of invertibility \cite[Theorem 2]{Bol}.
\end{remark}

In view of the above remark, Theorem \ref{t:main} implies that for any two Hermitian matrices $X,Y$ there exists a pair of commuting Hermitian matrices $X',Y'$ such that 
\beq
\label{main-matrices}
\|X-X'\|+\|Y-Y'\|\ \le\ C\,\|[X,Y]\|^{1/2}
\eeq
where the constant $C$ does not depend on $X$, $Y$ and the dimension. It also implies the following quantitative version of the BDF theorem.

\begin{cor}\label{c:bdf}
Assume that $H$ is separable and denote by $\K(H)$ the space of compact operators on $H$. If $A\in\B(H)$ and $d_1(A)=0$ then there exists a normal operator $A'\in\B(H)$ such that
\begin{equation}\label{main-bdf}
\begin{aligned}
&\|A-A'\|_{\ess} \le C\,\|[A,A^*]\|_{\ess}^{1/2}\,, \\
&\|A-A'\| \le C\left(\|[A,A^*]\|^{1/2}+\|A\|^{1/2}\|[A,A^*]\|_{\ess}^{1/4}\right),
\end{aligned}
\end{equation}
where $\|S\|_{\ess}:=\inf\limits_{K\in\K(H)}\|S-K\|$ is the essential norm and $C$ is a constant independent of $A$.
\end{cor}

\begin{remark}\label{r:bounds}
The upper bound \eqref{main-estimate1} is a difficult result, which is new even for finite matrices.
The lower bound \eqref{main-estimate2} is almost obvious and is proved in few lines (see Subsection \ref{s:lower-bound}). We have included it in Theorem \ref{t:main} only for the sake of completeness. 
The example $A_\eps=X+i\er Y$ with non-commuting Hermitian matrices $X,Y$ shows that $(5\|A\|)^{-1}\|[A,A^*]\|$ in \eqref{main-estimate2} can not be replaced by $C\|[A,A^*]\|^{1/2}$. On the other hand, if $A_\eps=\begin{pmatrix}\eps B & 0\\0 & I\end{pmatrix}$ and $[B,B^*]\ne0$ then the distance from $A$ to the set of normal matrices is estimated from above and below by $C\eps$ with some constants $C$  depending on $B$. This shows that the distance from $A$ to $\NN_f(\A)$ may decay as $\|[A,A^*]\|^{1/2}$ when $[A,A^*]$ tends to zero. 
\end{remark}

\subsection{Notation}
By the Gelfand--Naimark theorem, every $C^*$-algebra $\A$ can be isomorphically embedded into the algebra of bounded operators $\B(H)$ on a (not necessarily separable) Hilbert space $H$. In order to emphasize the operator-theoretic nature of our proofs, we shall always assume that $\A\subset\B(H)$ and refer to its elements as ``operators''. 

We shall use the following notation.
\begin{itemize}
\item $\O_r(\lambda):=\{z\in\C:|z-\lambda|<r\}$ is the open disc of radius $r$ about $\lambda$, and $\O_r:=\O_r(0)$.
\item $C$ and $\delta$ denote universal constants, which do not depend on the $C^*$-algebras and operators under consideration. 
\item $\sigma(A)$ is the spectrum of the operator $A$.
\item  $\M2(\A)$ is the $C^*$-subalgebra of $\B(H\oplus H)$ formed by $2\times 2$-matrices whose entries belong to $\A$. If $\A$ has real rank zero then  $\M2(\A)$ also has real rank zero (see, for instance, \cite{BP}).
\item
$P:=\begin{pmatrix}
I&0\\0&0
\end{pmatrix}\in \M2(\A)$
is the projection onto the first component of $H\oplus H$. 
\item
$\diag_P T:= PTP+(1-P)T(1-P)$ where $T\in M_2(\A)$.
\item 
$\GL_0(\A\oplus\A):=\l\{\begin{pmatrix}
A_1&0\\0&A_2
\end{pmatrix}\in \M2(\A)\colon A_1,A_2\in \GL_0(\A)\r\}$
\item
$d_2(T,\lambda):=\dist(T-\lambda I, \GL_0(\A\oplus \A))$ and $d_2(T):=\sup_{\lambda\in\C}d_2(T,\lambda)$, where $T\in M_2(\A)$. Note that
\beq
\label{intro1}
\|T-\diag_P T\|\ =\ \|[P,T]\|\ \le\ d_2(T,0)\,.
\eeq

\end{itemize}

\subsection{Plan of proof}

The first step is a generalization of \cite[Corollary 4.5]{D}. In Section \ref{s:ext} we prove that, under a certain condition on $A\in\A$, there exists a normal operator $T\in\M2(\A)$ such that 
$\|A-PTP\|\le C\,\|[A,A^*]\|^{1/2}$ and $d_2(T)\le d_1(A)+C\,\|[A,A^*]\|^{1/2}$.

The next step is the most difficult part of the proof.  In Theorem \ref{main_approx} we show that for every  normal operator $T\in\M2(\A)$ with a sufficiently small $d_2(T)$ there exists a normal operator $T_0\in\M2(\A)$ with finite spectrum such that $\|T-T_0\|\le3$ and $\|[P,T_0]\|\le C\,d_2(T)$. If $d_2(T)=0$ then Theorem \ref{main_approx} follows from \cite[Theorem 3.2]{FR2}. However, this does not help, since the operator $T$ constructed in the first step is not block diagonal and, consequently, $d_2(T)>0$.

Our proof of Theorem \ref{main_approx} uses the technique introduced in \cite{FR} and further developed in \cite{FS}. It is based on successive reductions of the operator to normal operators whose spectra do not contain certain subsets of the complex plane. One can think of this process as removing subsets from the spectrum. 

In order to obtain the uniform estimates for $\|T-T_0\|$ and $\|[P,T_0]\|$ in Theorem \ref{main_approx}, we use two auxiliary results proved in Section \ref{s:aux}. Their proof follows the same lines as that of Lemmas 4.1 and 4.2 in \cite{FS}, but with additional control of the commutator $[P,T_0]$.

Finally, in Section \ref{s:proofs} we adjust the block $PT_0P$ of the operator $T_0$ to obtain a normal operator lying within the prescribed distance from $A$. This yields Theorem \ref{t:main}. After that, we deduce Corollary \ref{c:bdf}  by combining \eqref{main-estimate1} with \cite[Theorem 3.8]{FS}. 

Throughout the paper, we shall be using various results on operator Lipschitz functions. Their statements and proofs are given in the next section.

\section{Operator Lipschitz functions}
\label{s:lip}

\begin{defi}
Let $\F\subset \C$ be a closed set.
A continuous function $f\colon \F\to \C$ belongs to the space $\OL(\F)$ if 
$$
\|f\|_{\OL(\F)}:=\sup\limits_{T_1,T_2}\frac{\|f(T_1)-f(T_2)\|}{\|T_1-T_2\|}<\infty\,,
$$
where the supremum is taken over all bounded normal operators $T_i$ acting on an infinite dimensional Hilbert space $H$ such that $\,\sigma(T_i)\subset\F\,$ and $T_1\ne T_2$.
\end{defi}

\begin{remark}
In the above definition, one can assume that the space $H$ is separable. Indeed, the $C^*$-algebra generated by two given operators $T_1$, $T_2$ and the identity operator is separable. Hence it is isomorphic to a subalgebra of $\B(H')$ for some fixed separable Hilbert space $H'$ (see, for instance, \cite[Theorem I.9.12]{Db}). It follows that for each pair of operators $T_1,T_2\in\B(H)$ there exist operators $T'_1,T'_2\in\B(H')$ such that $\|T_1-T_2\|=\|T_1'-T_2'\|$ and $\|f(T_1)-f(T_2)\|=\|f(T_1')-f(T_2')\|$.
\end{remark}

The spaces $\OL(\F)$ are complex quasi-Banach spaces, in which only constant functions have zero quasi-norms. The functions $f\in\OL(\F)$ are said to be {\it operator Lipschitz}. 
It is well known that an operator  Lipschitz function $f$ is commutator Lipschitz in the sense that 
\beq
\label{comm-l}
\|[X,f(N)]\|\ \le\ \|f\|_{\OL(\F)}\,\|[X,N]\|
\eeq
for all $N\in \NN(\B(H)))$ and $X=X^*\in \B(H)$
(see, for instance, \cite[Theorem 3.1]{AP3}). The best known sufficient conditions for the inclusion $f\in\OL(\F)$ are 
given in \cite{APPS,AP2,P} in terms of Besov spaces. For our purposes, it is sufficient to know that $C^2(\F)\subset \OL(\F)$ and $\|f\|_{\OL(\F)}\le C\,\|f\|_{C^2(\F)}$ for $\F=\R$ and $\F=\C$, where $\|f\|_{C^2(\F)}:=\max\limits_{0\le|\alpha|\le2}\sup\limits_{\zeta\in\F}|\partial_\zeta^\alpha f(\zeta)|$ and $C$ is some constant.

We shall need the following simple lemmas.

\begin{lemma}
\label{smoothmod}
If $\rho\in C(\R)$ then, 
for all $T_1,T_2\in \B(H)$,
\begin{align*}
\|\rho(T_1^* T_1)-\rho(T_2^* T_2)\| &\le \|\rho(t^2)\|_{\OL(\R)} \|T_1-T_2\|,\\
\|T_1 \rho(T_1^* T_1)-T_2 \rho(T_2^* T_2)\| &\le \|t\rho(t^2)\|_{\OL(\R)} \|T_1-T_2\|.
\end{align*}
\end{lemma}

\begin{proof}
Consider the self-adjoint operators $\,X_{j}=\begin{pmatrix} 0& T_{j}\\ T_{j}^*&0\end{pmatrix}\in\B(H\oplus H)$. Since
$$\rho(X_{j}^2)=\begin{pmatrix} \rho(T_{j}T_{j}^*) & 0\\ 0 & \rho(T_{j}^*T_{j})\end{pmatrix}\,\quad\text{and}\quad 
X_j \rho(X_{j}^2)=\begin{pmatrix} 0& T_{j}\rho(T_{j}^*T_{j})\\ T_{j}^* \rho(T_{j}T_{j}^*)&0\end{pmatrix},
$$
we obtain
\begin{multline*}
\|\rho(T_1^*T_1)-\rho(T_2^*T_2)\|\le \|\rho(X_1^2)- \rho(X_2^2)\|\\\le \|\rho(t^2)\|_{\OL(\R)}\|X_1-X_2\|=\|\rho(t^2)\|_{\OL(\R)}\|T_1-T_2\|
\end{multline*}
and, similarly,
$$
\|T_1 \rho(T_1^*T_1)-T_2 \rho(T_2^*T_2)\|\le \|X_1 \rho(X_1^2)-X_2 \rho(X_2^2)\|\le \|t \rho(t^2)\|_{\OL(\R)}\|T_1-T_2\|.\,\,\qedhere
$$
\end{proof}

\begin{lemma}
\label{manysupp}
Suppose that $\chi\in C_0^{\infty}(\C)$, $0\le \chi\le 1$, and let $\lambda_1,\lambda_2,\ldots,\lambda_k\in\C$ be a finite collection of points 
such that $\chi_i\chi_j\equiv0$ for all $i\neq j$, where $\chi_j(z):=\chi(z-\lambda_j)$.  Let also
$M:=\sum_{j=1}^k\chi_j(T)S_j \chi_j(T)$, where $T\in\NN(\B(H))$ and $S_j\in\B(H)$ . If $\|S_j\|\le 1$ and $[S_j,\chi_j(T)]=0$ for all $j=1,\ldots,k$ then
$$
\|[Q,M]\|\ \le\ C_\chi\,\|[Q,T]\|\,\max_j \|S_j \chi_j(T)\|+\max_j \|[Q,S_j]\|
$$
for all self-adjoint $Q\in\B(H)$, where $C_\chi$ is a constant depending only on $\chi$.
\end{lemma}

\begin{proof}
We have $[Q,M]=R_1+R_2+R_3\,$, where $\,R_1=\sum_{j=1}^k[Q,\chi_j(T)] S_j \chi_j(T)\,$, $\,R_2=\sum_{j=1}^k \chi_j(T)
[Q,S_j] \chi_j(T)\,$ and $\,R_3=\sum_{j=1}^k \chi_j(T)S_j [Q,\chi_j(T)]\,$. 

Denote by $\hat\chi$ the Fourier transform of $\chi(x+iy)$ as a function of two real variables $x$ and $y$. Let $\lambda_j=x_j+iy_j$ and $T=X+iY$, where $X,Y$ are self-adjoint. Then $[X,Y]=0$ and
$$
R_1\
=\ \frac{1}{4\pi^2} \int_{\R^2}\hat \chi(s, t)\,[Q,e^{isX+itY}]\,
\Bigl\{\sum\limits_{j=1}^k e^{-is x_j-it y_j}S_j \chi_j(T)\Bigr\}\,\dr s\,\dr t.
$$
Since the operators $S_j \chi_j(T)$ act in mutually orthogonal subspaces, the norm of the sum in curly brackets does not exceed 
$\max_j \|S_j \chi(T-\lambda_j I)\|$. Also, $\|[Q,e^{isX+itY}]\|\le(|s|+|t|)\,\|[Q,T]\|\,$ because
$$
\|[Q,e^{isX}]\|=\|[e^{-isX}Qe^{isX}-Q\|=\Bigl\|\int_0^se^{-itX}[Q,X]e^{itX}\dr t\Bigr\|\le|s|\,\|[Q,X]\|
$$
and, similarly, $\|[Q,e^{itY}]\|\le|t|\,\|[Q,Y]\|$. It follows that  
\begin{equation}
\label{medium}
\|R_1\|\
\le\ \frac{\|[Q,T]\|}{4\pi^{2}}\,\max_j \|S_j \chi_j(T)\|\int_{\R^2}(|s|+|t|)\,|\hat\chi(s,t)|\,\dr s\,\dr t.
\end{equation}
Similar arguments show that $\|R_3\|$ admits the same estimate. It remains to notice that $\|R_2\|\le \max_j \|[Q,S_j]\|$ because $\chi_i\chi_j\equiv0$ for  $i\neq j$.
\end{proof}

\begin{lemma}\label{diagonalize}
Let $\rho\in C_0^{\infty}(\R)$ be a nonnegative function such that  $\,\supp \rho\subset [-1,1]\,$ and $\,\sum_{n\in \Z}\rho_n^2(x)=1$, where $\rho_n(x)=\rho(x-n)$. If $X,Y\in\B(H)$ are self-adjoint  and $\,Y'=\sum_{n\in \Z} \rho_n(X)Y \rho_n(X)\,$ then
\begin{enumerate}
\item[\rm(d$_1$)]
$\|EY'E\|\le \|EYE\|$ for any spectral projection $E$ of the operator $X$;
\item[\rm(d$_2$)]
$\|[X,Y']\|\le \|[X,Y]\|\,$;
\item[\rm(d$_3$)]
$\|Y-Y'\|\le C_\rho\|[X,Y]\|$ where $C_\rho$ is a constant depending only on $\rho$.
\end{enumerate}
\end{lemma}

\begin{proof}
If $aI\le EYE\le bI$ then $a\|u\|^2\le(EY'Eu,u)\le b\|u\|^2$ for all $u\in H$. This implies (d$_1$). Since the commutator $[X,Y]$ is skew-adjoint, the inequality (d$_2$) follows from the identity 
$\,[X,Y']=\sum_{n\in \Z} \rho_n(X)[X,Y]\rho_n(X)$
by similar arguments. Finally, since
$Y-Y'=\sum_n [Y,\rho_n(X)]\rho_n(X)$, the estimate (d$_3$) is proved in the same way as \eqref{medium}.
\end{proof}

\section{An extension theorem}
\label{s:ext}

The following theorem is a refined version of \cite[Corollary 4.5]{D}. The latter provides normal operators $N\in\B(H)$ and $T\in\B(H\oplus H)$ satisfying (e$_1$) with $C=50\sqrt2$ and (e$_2$). If $\A=\B(H)$ and $\dim H<\infty$ then (e$_3$) follows from (e$_1$), so that  \cite[Corollary 4.5]{D} is sufficient to prove \eqref{main-matrices}.

\begin{theorem}\label{davidson}
Let $\A$ be a unital $C^*$-algebra, and let $A\in\A$. If $\,\re A\,$ can be approximated by self-adjoint operators from $\A$ with finite spectra then there exist normal operators $N\in\A$ and $T\in M_2(\A)$ such that 
\begin{enumerate}
\item[\rm(e$_1$)]
$\|A\oplus N-T\|\le C\,\|[A,A^*]\|^{1/2}$,
\item[\rm(e$_2$)]
$\|N\|\le\|A\|$ and $\|T\|\le\|A\|$,
\item[\rm(e$_3$)]
$d_2(T)\le d_1(A)+C\,\|[A,A^*]\|^{1/2}$,
\end{enumerate}
where $C$ is a constant independent of $\A$ and $A$.
\end{theorem}

\begin{proof}
Let $A=X+iY$, where $X,Y$ are self-adjoint.
First of all, let us make some reductions. Note that the statements (e$_1$)--(e$_3$) are invariant under multiplication of $A$ by a scalar. Therefore, without loss of generality, we shall be assuming that $\|[A,A^*]\|=1$. Since $X$ is approximated by operators with finite spectra, we can also assume that $X$ has finite spectrum. 
Finally, (e$_2$) can be replaced with the weaker condition
\begin{enumerate} 
\item[\rm(e$'_2$)]
$\|N\|\le\|A\|+1$.
\end{enumerate}
Indeed, since $\|[A,A^*]\|=1$, from the estimates (e$_1$) and (e$'_2$) it follows that 
\begin{enumerate} 
\item[\rm(e$''_2$)]
$\|T\|\le \|A\|+C+1$.
\end{enumerate}
If $\,c=\frac{\|A\|}{\|A\|+C+1}\,$ then, in view of (e$_1$), (e$'_2$) and (e$''_2$),
\begin{multline*}
\|A\oplus cN-cT\|\le\|A\oplus N-T\|+(1-c)\left(\|N\|+\|T\|\right)\\
\le C+2(1-c)\,(\|A\|+C+1)=3C+2\,,
\end{multline*}
$\|cN\|\le\|A\|$, $\|cT\|\le\|A\|$ and
$\,d_2(cT)=c\,d_2(T)\le d_2(T)\,$.
Thus we see that (e$_1$)--(e$_3$) hold with the constant $C$ replaced by $3C+2$ for the normal operators $cN$ and $cT$.

Let us fix a function $\rho$ satisfying the conditions of Lemma \ref{diagonalize}, and denote by $E_n$ the spectral projections of $X$ corresponding to the intervals $[n;n+1)$. Since $\sigma(X)$ is finite, the projections $E_n$ belong to $\A$. Consider the operators 
$$
\Pi_n\ :=\ 
\begin{pmatrix}
\rho_n^2(X) & \psi_n(X)\\
\psi_n(X) & E_{n-1}+E_n-\rho_n^2(X)
\end{pmatrix}
\ \in\ M_2(\A)\,,
$$
where $\rho_n(x)=\rho(x-n)$ and
$\psi_n(x)=(-1)^n\rho_n(x)(\rho_{n-1}(x)+\rho_{n+1}(x))$. One can easily see that $\Pi_n$ are mutually orthogonal projections such that 
\begin{equation}\label{dav3}
\sum_{n\in \Z}\Pi_n=\begin{pmatrix}I & 0\\ 0 & I\end{pmatrix}\quad\text{and}\quad
\Pi_n=\begin{pmatrix}E_{n-1}+E_n & 0\\ 0 & E_{n-1}+E_n\end{pmatrix}\Pi_n
\end{equation}
Note that $[\Pi_n,E_m\oplus E_m]=0$ for all $n,m\in \Z$. It follows that 
$(E_{n}\oplus E_n)\Pi_m$ are mutually orthogonal projections such that 
\beq
\label{smallsplit}
\sum\limits_{n,m\in\Z}\begin{pmatrix}E_n & 0\\ 0 & E_n\end{pmatrix}\Pi_m=\sum\limits_{|n-m|\le 1}\begin{pmatrix}E_n & 0\\ 0 & E_n\end{pmatrix}\Pi_m=\begin{pmatrix}I & 0\\ 0 & I\end{pmatrix}.
\eeq

Let $Y'$ be the operator defined in Lemma \ref{diagonalize}, and let $Y''=\sum_{m\in \Z} E_m Y' E_m$.
We claim that the normal operators 
$$
N:=\sum_{n\in \Z}E_n\left(nI +i Y'\right)E_n
\quad\text{and}\quad
T:=\sum_{n\in \Z}\Pi_n\begin{pmatrix}nI+i Y' & 0\\ 0 & nI+iY''\end{pmatrix}\Pi_n
$$
satisfy the conditions (e$_1$), (e$'_2$) and (e$_3$). 

First, let us prove (e$_1$). Since $E_n (X-nI) E_n\le 1$, using \eqref{smallsplit}, we obtain
\begin{multline}
\label{estimate1}
\|\re (A\oplus N)-\re T\|\ \le\ 1+\Bigl\|\sum\limits_{n}n (E_n\oplus E_n-\Pi_n)\Bigr\|\\
=\ 1+\Bigl\|\sum\limits_{n,m}\left(n(E_n\oplus E_n)\Pi_m-n(E_m\oplus E_m)\Pi_n\right)\Bigr\|
\ \\=\
1+\Bigl\|\sum\limits_{|n-m|\le 1}(n-m)(E_n\oplus E_n)\Pi_m\Bigr\|\ \le\ 2.
\end{multline}
For the imaginary part, in view of the estimate (d$_3$) in Lemma \ref{diagonalize}, we have 
\beq
\label{estimate2}
\|\im(A\oplus N)-Y'\oplus Y''\|\ =\ \|Y\oplus Y''-Y'\oplus Y''\|\ \le\ C_\rho\,,
\eeq
where $C_\rho$ is a constant depending only on the choice of $\rho$. Note that $E_n Y'E_m=0$ whenever $|n-m|\ge 2$.
These identities and the second equality \eqref{dav3} imply that $\,\im T=\sum_{|n-m|\le 2}\Pi_n (Y'\oplus Y'')\Pi_m\,$ and, consequently,
\begin{multline}
\label{estimate3}
\l\|Y'\oplus Y''-\im T\r\|\ =\ \Bigl\|\sum\limits_{1\le |n-m|\le 2}\Pi_n (Y'\oplus Y'')\Pi_m\Bigr\|\\ 
\le\ 4\max_{1\le |n-m|\le 2}\l\|\Pi_n (Y'\oplus Y'')\Pi_m\r\|\ \le\ 4\max_n \|[\Pi_n,Y'\oplus Y'']\|.
\end{multline}
Since $[E_n,Y'']=[E_{n-1},Y'']=0$, we have
\beq
\label{matrix1}
[\Pi_n,Y'\oplus Y'']=
\l(\begin{matrix}
[\rho^2_n(X),Y']& \psi_n(X)Y''-Y'\psi_n(X)\\
\psi_n(X)Y' - Y''\psi_n(X) & [Y'',\rho^2_n(X)]
\end{matrix}\r).
\eeq
The operator \eqref{matrix1} is skew-adjoint and
\begin{multline}\label{off-diag}
\psi_n(X)Y''-Y' \psi_n(X)\\
=(-1)^n\l([\rho_n(X)\rho_{n+1}(X),Y'] E_n+[\rho_n(X)\rho_{n-1}(X),Y'] E_{n-1}\r).
\end{multline}

By Lemma \ref{diagonalize}, $\|[X,Y']\|\le\|[X,Y]\|=\frac12\,\|[A,A^*]\|=1/2$ and, similarly, 
$$
\|[X,Y'']\|\ =\ \Bigl\|\sum_{m\in\Z}E_m[X,Y']E_m\Bigr\|\ \le\ \|[X,Y']\|\ \le\ 1/2\,.
$$
Since $C^2$-norms of $\rho_n^2$ and $\rho_n\rho_{n+1}$ are estimated by a constant independent of $n$, 
the inequality \eqref{comm-l} and \eqref{off-diag} imply that the norms of all entries in the right hand side of \eqref{matrix1} are estimated by constants depending only on $\rho$. Together with \eqref{estimate1}, \eqref{estimate2} and \eqref{estimate3}, this yields (e$_1$).

Obviously, $\|E_n (X-nI) E_n\|\le 1$. The estimate (d$_1$) in Lemma \ref{diagonalize} implies that $\,\|E_n(nI+iY')E_n\|\le\|E_n(nI+iY)E_n\|\,$. It follows that 
$$
\|N\|=\max_n\|E_n(nI+iY')E_n\|\le\max_n\|E_n(nI+iY)E_n\|\le \|A\|+1.
$$
Finally, $\sigma(N)$ is a bounded subset of $\Z+i\R$. By Remark \ref{gl-0}, $d_1(N)=0$. This equality and the estimate (e$_1$) imply (e$_3$).
\end{proof}

\section{Two auxiliary theorems}
\label{s:aux}

Recall that any invertible operator $T$ has the polar decomposition $T=V|T|$, where $|T|=\sqrt{T^*T}$ and $V=T|T|^{-1}$ is a unitary operator from the same $C^*$-algebra as $T$. A normal operator $T$ also admits the polar decomposition $T=V|T|$ with a unitary $V$. However, if $T$ is not invertible then, generally speaking, the unitary polar part $V$ of $T$ is not uniquely defined and may not belong to the same $C^*$-algebra. In the both cases, the unitary operator $V$ satisfies $V\rho(|T|)=\rho(|T^*|)V$ for all $\rho\in C(\R^1)$. If $T$ is normal, this implies that $V$ commutes with all continuous functions of $T$.

In the next two theorems $V$ is a unitary polar part of $T$ and $\Pi_{r}$ is the spectral 
projection of $T$ onto the disc $\O_r\,$.

\begin{theorem}
\label{t:u1}
There exist constants $\delta>0$ and $C>0$ such that 
for any normal $T\in \M2(\A)$ with 
$d_2(T,0)<\delta$ one can find a unitary operator $U\in \M2(\A)$ satisfying the following conditions.
\begin{enumerate}
\item[{\rm(u$_1$)}] $\diag_P U \in \GL_0(\A\oplus \A)$,
\item[{\rm(u$_2$)}] $[U,\Pi_{r}]=0$ for $r\ge 1$,
\item[{\rm(u$_3$)}] $U(I-\Pi_{1})=V(I-\Pi_{1})$,
\item[{\rm(u$_4$)}] $\|[P,U]\|\le C d_2(T,0)$.
\end{enumerate}
\end{theorem}

\begin{theorem}
\label{t:u2}
There exist constants $\delta>0$ and $C>0$ such that the following is true. Let $\A$ have real rank zero, and let $T\in \M2(\A)$ be a normal operator such that $\sigma(T)\cap \O_3$ is a subset of the straight line 
$e^{i\theta}\R$, where $\theta\in[0,\pi)$. If $\|[P,T]\|<\delta$ and $\,\diag_P(T\pm ie^{i\theta} I)\in\GL_0(\A\oplus\A)$ then
one can find a unitary operator $U\in \M2(\A)$ satisfying {\rm(u$_1$)} and the following conditions.
\begin{enumerate}
\item[{\rm(u$'_2$)}] $[U,\Pi_{r}]=0$ for $1\le r\le 2$,
\item[{\rm(u$'_3$)}] $U(\Pi_2-\Pi_{1})=V(\Pi_2-\Pi_{1})$,
\item[{\rm(u$'_4$)}] $\|[P,U]\|\le C \|[P,T]\|$,
\item[{\rm(u$'_5$)}] 
the spectrum of $\left.U\right|_{\Ran\Pi_2}$ is contained in the intersection of 
$e^{i\theta}\R$ with the unit circle.
\end{enumerate}
\end{theorem}

The proofs of Theorems \ref{t:u1} and \ref{t:u2} use the following lemmas.

\begin{lemma}
\label{fsl2}
Suppose that $\A$ is a $C^*$-algebra of real rank zero, and let $U\in \GL_0(\A)$ be unitary. 
Then for any $\eps>0$ there exists a unitary operator $V_\eps\in \GL_0(\A)$ such that $-1\notin\sigma(V_\eps)$ and $\|U-V_\eps\|\le\eps$.
\end{lemma}

Lemma \ref{fsl2} is contained in \cite{L1}. See also \cite[Lemma 1.8]{FS} 
for a more elementary proof.

\begin{lemma}
\label{gl0}
Let $t\mapsto G_t$ be a continuous path in $\GL(\M2(\A))$ such that $\|G_t^{-1}\|<\|[P,G_t]\|^{-1}$ for all $t\in[0,1]$ and $\diag_P G_0\in \GL_0(\A\oplus \A)$. Then $\diag_P G_1\in \GL_0(\A\oplus \A)$.
\end{lemma}

\begin{proof}
Since $\diag_P G_t=G_t\left(I-G_t^{-1}(G_t-\diag_P G_t)\right)$ and $\|G_t-\diag_P G_t\|=\|[P,G_t]\|$,
the operators $\diag_P G_t$ are also invertible. Hence, the path $t\mapsto \diag_P G_t$ connects $G_0$ and $G_1$ in $\GL(\A)\oplus\GL(\A)$. As $G_0\in \GL_0(\A\oplus \A)$, so does $G_1$.
\end{proof}

\begin{lemma}
\label{l:radial}
If there exists a unitary operator $V$ such that $\|S-V\|\le\eps<1$ then the operator $S$ is invertible and $\|S-S|S|^{-1}\|<\frac{\eps\,(1+\eps)}{1-\eps}\,$.
\end{lemma}

\begin{proof}
If $S=V+R$ with $\|R\|\le\eps$ then $(1-\eps)^2I\le S^*S\le(1+\eps)^2I$. It follows that $S$ is invertible. Since $(1+\eps)^{-1}I\le |S|^{-1}\le(1-\eps)^{-1}I$ and
$\|S\|\le(1+\eps)$, we obtain $\|S-S|S|^{-1}\|\le(1+\eps)\left\|I-|S|^{-1}\right\|\le(1+\eps)\,\eps(1-\eps)^{-1}$.
\end{proof}

\begin{lemma}
\label{contour_cut}
Suppose that $\Gamma$ is a simple closed curve given by an equation of the form
\beq
\label{Gamma}
\Gamma\ =\ \{z\in\C:|z-\lambda|=\ph\left(\arg(z-\lambda)\right)\}
\eeq
where $\lambda$ is an interior point of the domain bounded by $\Gamma$ and $\ph\in C^2(\R)$ is a strictly positive $2\pi$-periodic function. 
Let $\A$  be a unital $C^*$-algebra of real rank zero, and let $T\in \M2(\A)$ be a normal operator such that $\,\sigma(T)\subset \Gamma\,$ and $\,\diag_P (T-\lambda I)\in \GL_0(\A\oplus \A)\,$. Then for every $z_0\in\Gamma$ there exists 
a normal operator $T_0$ such that $\sigma( T_0)\subset \Gamma\setminus\{z_0\}$, $T_0-\lambda I\in \GL_0(\A\oplus \A)$ and 
$\| T-T_0\|\le C_\Gamma\,\|[P,T]\|$, where $C_\Gamma$ is a constant depending only on $\Gamma$.
\end{lemma}

\begin{proof}
Without loss of generality, we can assume that $\lambda =0$ and $z_0\in\R_-$. 
Also, it is sufficient to prove the lemma assuming 
that $\|[P,T]\|$ is small enough, since we have $\| T_0-T\|\le C_\Gamma\,\|[P,T]\|$ for 
any normal operator $T_0$ with $\sigma(T_0)\subset\Gamma$ if $\|[P,T]\|\ge\eps$ and $C_\Gamma=2\eps^{-1}\diam\Gamma$.

Let $\ph_t(z)=z\left(t\ph(\arg z)+1-t\right)^{-1}$. The functions $\ph_t$ belong to $C^2$ on an annulus $\Omega$ containing $\Gamma$, and their $C^2(\Omega)$-norms are bounded by a constant depending only on $\Gamma$. Obviously, they can be extended to $C^2$-functions on $\C$ whose $C^2$-norms admit a similar estimate.

The operator $\ph_1(T)$ is unitary because $\ph_1$ maps $\Gamma$ onto the unit circle. By \eqref{comm-l}, we have $\|[\ph_t(T),P]\|\le C'_\Gamma\,\|[P,T]\|$ for all $t\in[0,1]$ with a constant  $C'_\Gamma$ depending only on $\Gamma$. If $\|[P,T]\|$ is sufficiently small, Lemma \ref{gl0} with $G_t=\ph_t(T)$ implies that $S:=\diag_P \ph_1(T)$ belongs to $\GL_0(\A\oplus \A)$. 

Since $\|S-\ph_1(T)\|=\|[\ph_1(T),P]\|\le C'_\Gamma\,\|[P,T]\|$, the operator $S$ is close to the unitary operator $\ph_1(T)$. If $U=S|S|^{-1}$ then $U\in\GL_0(\A\oplus \A)$ and, by Lemma \ref{l:radial}, $\|S-U\|\le C''_\Gamma\,\|[P,T]\|$ where  $C''_\Gamma$ depends only on $\Gamma$.

Now, applying Lemma \ref{fsl2}, we find a unitary $U_0\in\GL_0(\A\oplus \A)$ such that $\|U-U_0\|\le\|[P,T]\|$ and $-1\notin\sigma(U_0)$, and define $T_0=\phi_1^{-1}(U_0)$ where $\ph_1^{-1}(z)=z\left(\ph(\arg z)\right)$ is the inverse function. Since $\varphi_1^{-1}$ maps the unit circle onto $\Gamma$ and $\phi_1(-1)=z_0$, we have $\sigma( T_0)\subset \Gamma\setminus\{z_0\}$. Since $\ph_1^{-1}$ belongs to $C^2$ on a neighbourhood of the unit circle, it can be extended to an operator Lipshitz function on $\C$. Therefore the inequality $\| T-T_0\|\le C_\Gamma\,\|[P,T]\|$ follows from the estimate
$$
\|\ph_1(T)-U_0\|\le\|\ph_1(T)-S\|+\|S-U\|+\|U-U_0\|\le\left(C'_\Gamma+C''_\Gamma+1\right)\|[P,T]\|.
$$
Finally, $T_0\in \GL_0(\A\oplus \A)$ because the complement of its spectrum is a dense connected set (see Remark \ref{gl-0}).
\end{proof}

\begin{remark}
One can easily extend Lemma \ref{contour_cut} to a much wider class of curves $\Gamma$, but \eqref{Gamma} will be sufficient for our purposes. 
\end{remark}

\subsection{Proof of Theorem \ref{t:u1}}
Let $V\in \B(H\oplus H)$ be a unitary operator such that $T=V|T|$. 
Let us denote 
$\delta_1:=d_2(T,0)$ and choose $T_0\in \GL_0(\A\oplus \A)$ such that $\|T-T_0\|\le 2\delta_1$. We have $T_0=V_0|T_0|$ with 
$|T_0|\in \GL_0(\A\oplus \A)$ and a unitary $V_0\in \GL_0(\A\oplus \A)$.

Let $\rho_1\in C^{\infty}(\mathbb R_+)$ be a nonincreasing function such that $\rho_1(t)=1$ for $t\in[0,\frac12]$ and $\rho_1(t)=0$ for $t\ge 1$. Define $\,\rho_2:=\sqrt{1-\rho_1^2}\,$ and consider the operator
\beq
\label{sm}
S=\rho_1(T^*T)V_0\rho_1(T^*T)+V \rho_2^2(T^*T).
\eeq
We have 
$V\rho_2^2(T^*T)=\chi(T)$, where 
$\chi(z)=z|z|^{-1}\rho_2^2(|z|^2)$ is a $C^\infty$-function. Thus 
$S\in \M2(\A)$. Since $\rho_1(|z|^2)\in\OL(\C)$, $\chi(z)\in\OL(\C)$ and $[P,V_0]=0$, 
from 
\eqref{comm-l} and \eqref{intro1} it follows that
$$
\|[P,S]\|\le 2\|[P,\rho_1(T^*T)]\|+\|[P,\chi(T)]\|\le 
C_1\|[P,T]\|\le C_1\,\delta_1.
$$ 
Here and in the rest of the proof $C$ with a subscript denotes a constant depending only on the choice of $\rho_1$. 

By Lemma \ref{smoothmod},
\begin{align*}
\|\rho_1(T^*T)-\rho_1(T_0^* T_0)\| &\le C_2\|T-T_0\|\le 2C_2\,\delta_1,\\
\|\rho_1(TT^*)-\rho_1(T_0 T_0^*)\| &\le C_2\|T-T_0\|\le 2C_2\,\delta_1.
\end{align*}
Since $TT^*=T^*T$, these estimates and the identity $\rho_1(T_0 T_0^*)V_0=V_0\rho_1(T_0^* T_0)$ imply that
\beq
\label{svo1}
\|S-V_0 \rho_1^2(T_0^* T_0)-V \rho_2^2(T^*T)\|\le C_3\,\delta_1.
\eeq
The function $\rho_2(t)$ vanishes in a neighbourhood of zero. Hence $\rho_2^2(t^2)=t \psi(t^2)$ with a smooth bounded function $\psi$ and, by Lemma \ref{smoothmod},
\begin{multline}
\label{svo2}
\|V_0 \rho_2^2(T_0^* T_0)-V\rho_2^2(T^* T)\|=\|V_0 |T_0|\psi(T_0 ^* T_0)-V |T| \psi(T^* T)\|\\=
\|T_0 \psi(T_0^* T_0)-T \psi(T^* T)\|\le C\|T-T_0\|\le 2C\delta_1.
\end{multline}
Combining \eqref{svo1} with \eqref{svo2} and using the identity $\rho_1^2+\rho_2^2\equiv1$, we obtain
$\|S-V_0\|\le C_4\,\delta_1$.

Let us assume that $\delta$ in the statement of the lemma is so small that $C_4\delta_1<1$. If $U=S|S|^{-1}$ then, by Lemma \ref{l:radial}, 
\beq
\label{est2}
\|U-V_0\|\le \|U_{}-S_{}\|+\|S_{}-V_0\|\le  C_5\,\delta_1.
\eeq

If $\delta$ is small then the spectrum of $S^*S_{}$ lies in a small neighbourhood of 1.  Hence the operator $|S|^{-1}$ can be expressed as a smooth function of $S_{}^*S_{}$ 
supported on a small interval containing 1.
As $\|[P,S]\|\le C_1\,\delta_1$, 
we have $\|[P,|S|^{-1}]\|\le C_6\,\delta_1$. These two estimates imply (u$_4$). 

Let $G_t=(1-t)V_0+tU$. If $C_5\,\delta<\frac13$ then, in view of \eqref{est2}, $\|[P,G_t]\|<\frac23$ and $\|G_t^{-1}\|<\frac32$ for all $t\in[0,1]$. Applying Lemma \ref{gl0}, we obtain (u$_1$). 

Finally, we have (u$_2$) and (u$_3$) because $[S,\Pi_r]=0$ and $\rho_1(T^*T)\Pi_r=\rho_1(T^*T)=\Pi_r\rho_1(T^*T)$ for all $r\ge1$.
 \qed

\subsection{Proof of Theorem \ref{t:u2}}
Multiplying $T$ by a constant, we can assume that $\theta=0$ and, consequently, $\sigma(T)\cap \O_3\subset\R$. The proof consists of two parts. 

\subsubsection{Part 1}
Suppose first that $\sigma(T)\subset\Gamma$, where $\Gamma$ is a closed curve of the form  \eqref{Gamma} containing the line segment $[-3,3]$.

Denote $\delta_1:=\|[P,T]\|$. By Lemma \ref{contour_cut}, there exists a normal operator $T_0\in \GL_0(\A\oplus \A)$ 
such that $\sigma(T_0)\subset \Gamma\setminus \{0\}$ and $\|T-T_0\|\le C_\Gamma\,\delta_1$. We have  $T_0=V_0|T_0|$ with 
$|T_0|\in \GL_0(\A\oplus \A)$ and a unitary $V_0\in \GL_0(\A\oplus \A)$. 

Let $\rho_1$, $\rho_2$ and $V$ be as in the proof of Theorem \ref{t:u1}. Consider the operator
$$
S=\rho_1(T^*T)(\re V_0)\rho_1(T^*T)+V \rho_2(T^*T)^2.
$$
Since $\rho_1(T^*_0T_0)(\re V_0)\rho_1(T^*_0T_0)=V_0\rho^2_1(T^*_0T_0)$,
the same arguments as in the proof of Theorem \ref{t:u1} show that
$$
\|S-V_0 \rho_1(T_0^* T_0)^2-V \rho_2(T^*T)^2\|\le C_{\Gamma,1}\,\delta_1
$$
and $\|S-V_0\|\le C_{\Gamma,2}\delta_1$, where $C_{\Gamma,j}$ are constants depending only on $\Gamma$. 

If $U=S|S|^{-1}$ then, in the same way as in the proof of Theorem \ref{t:u1}, we obtain (u$_1$)--(u$_3$) and (u$'_4$) with some constant $C_\Gamma$ depending only on $\Gamma$. The condition (u$'_5$) is fulfilled because the operator $\Pi_2S\Pi_2$ is self-adjoint, and so is the unitary operator $\left.U\right|_{\ran\Pi_2}$. 

\begin{remark}\label{r:part1}
Note that 
\begin{enumerate}
\item[(a)]
the above proof works under the weaker assumption $\sigma(T)\cap \O_2\subset(-2,2)$;
\item[(b)]
we only need to check that $\,\diag_P(T_j-\lambda I)\in\GL_0(\A\oplus\A)$ for one point $\lambda$ lying in the domain bounded by $\Gamma$;
\item[(c)]
the constructed operator $U$ satisfies stronger the conditions (u$_2$) and (u$_3$), which imply (u$'_2$), (u$'_3$). 
\end{enumerate}
\end{remark}

\subsubsection{Part 2} 
Let now $\sigma(T)$ be an arbitrary closed bounded set such that $\sigma(T)\cap \O_3\subset(-3,3)$. First of all, let us note that the statement of Theorem \ref{t:u2} is local in the following sense. 

Assume that we have found a unitary $U_j$ satisfying the conditions of Theorem \ref{t:u2} for another normal operator $T_j$. If 
\begin{enumerate}
\item[(i)]
the spectral projection of $T_j$ corresponding to $\O_2$ coincides with $\Pi_2$ and $T_j\Pi_2=T\Pi_2$,
\item[(ii)]
$\|[P,T_j]\|\le C_j\,\|[P,T]\|$ where the constant $C_j$ does not depend on $\A$ and $T$,
\item[(iii)] 
$\,\diag_P(T_j\pm iI)\in\GL_0(\A\oplus\A)$,
\end{enumerate}
then (u$_1$) and (u$'_2$)--(u$'_5$) hold for $T$ and $U=U_j$. 

Hence, in view of Part 1, it is sufficient to find a normal operator $T_j$ satisfying the conditions (i)--(iii) whose spectrum lies on a curve of the form \eqref{Gamma} with $\lambda=i$. Furthermore, in view of Remark \ref{r:part1}(b), if $\sigma(T_j)\subset\Gamma$ then we have to prove the inclusion (iii) only for $\diag_P(T_j-iI)$.
The construction proceeds in several steps. 

\begin{figure}[H]
\begin{center}
\includegraphics[scale=0.7]{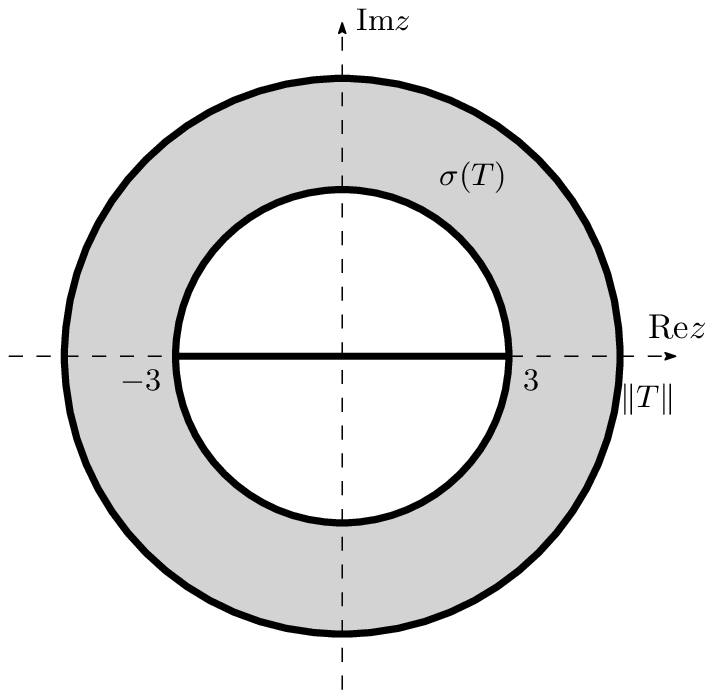}
\includegraphics[scale=0.7]{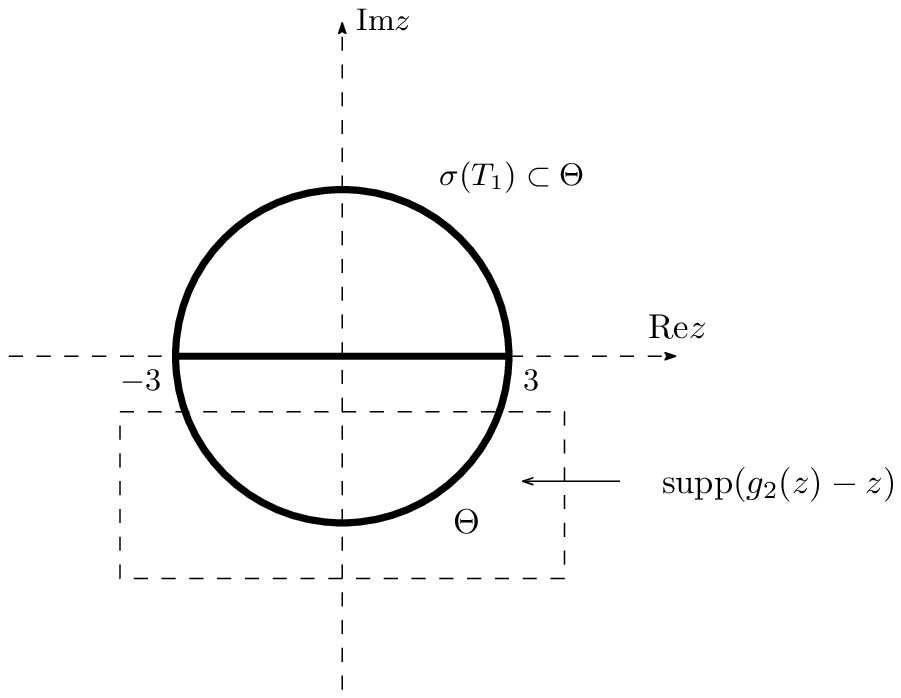}
\end{center}
\caption{The spectra of $T$ and $T_1=g_1(T)$}
\label{zz}
\end{figure}

First, let us choose a function $g_1\in C^2(\C)$ such that 
$g_1(z)=3z/|z|$ for $|z|\ge 3$, $g_1(z)=z$ for $|z|\le 2$, and $g_1(z)/z>0$ for all $z\neq 0$. Put $T_1=g_1(T)$. Then $\sigma(T_1)\subset\Theta=(-3;3)\cup\partial\O_3$, see Figure \ref{zz}. Clearly, $T_1$ satisfies (i). In view of \eqref{comm-l}, we also have (ii). Finally, if $\delta$ is small enough then  the paths $G_t=t g_1(T)+(1-t)T\pm iI$ satisfy the assumptions of Lemma \ref{gl0}, which implies the inclusions (iii).

Let $g_2$ be a smooth function mapping the arc of $\Theta$ between $(-3-3i)/\sqrt{2}$ and $(3-3i)/\sqrt{2}$ into the line segment $[-2-2i;2-2i]$ such that $g_2(z)=z$ outside the lower rectangle in the right part of 
Figure \ref{zz}. We have $g_2\in \OL(\C)$ since it is a smooth compactly supported perturbation of $z$. Let $\Theta_2=g_2(\Theta)$ and $T_2=g_2(T_1)$, so that $\sigma(T_2)\subset\Theta_2$. For the same reasons as before, the operator $T_2$ satisfies (i)--(iii).

There is a function $g_3\in\OL(\C)$ such that $g_3(z)=z$ outside the upper rectangle of Figure \ref{f2}, $g(z)-z\in C^2(\C)$ and $g_3(\Theta_2)$ is a simple curve given by an equation of the form \eqref{Gamma}.
Note that $g_3\colon\C\to \C$ is not a 
diffeomorphism, as $g_3$ maps two arcs of $\Theta_2$ into one.
Let $T_3=g_3(T_2)$. The same arguments as for $T_1$ show that $T_3$ satisfies (ii) and $\diag_P T_3-iI\in \GL_0(\A\oplus \A)$.

\begin{figure}[H]
\begin{center}
\includegraphics[scale=0.7]{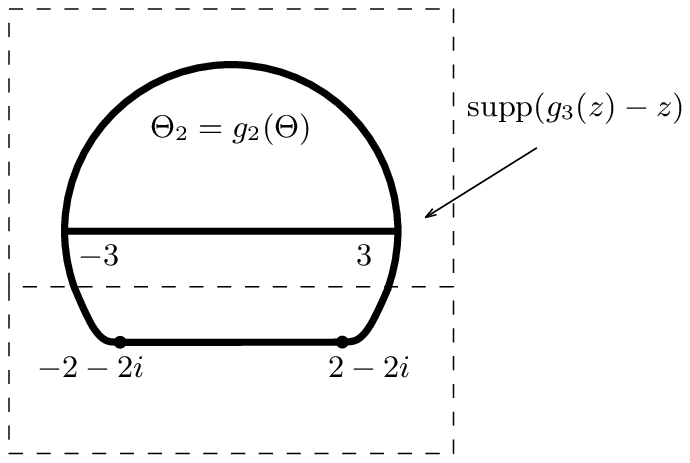}
\includegraphics[scale=0.7]{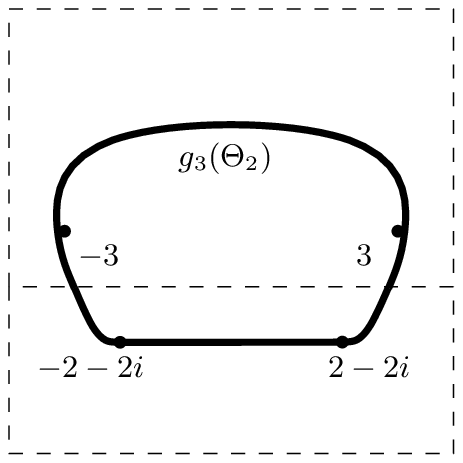}
\end{center}
\caption{The spectra of $T_2=g_2(T_1)$ and $T_3=g_3(T_2)$}
\label{f2}
\end{figure}

Applying Part 1 to the operator $T_3+2iI$, we obtain a unitary operator $U_3$ satisfying (u$_1$)--(u$_3$), (u$'_4$) and (u$'_5$) with $\Pi_r$ being the spectral projections of $T_3+2iI$. Note that $g_2(z)\equiv z$ on  $\O_2(-2i)\cap\sigma(T_2)$. Therefore $\Pi_r$ coincide with the spectral projections $\Pi_r(-2i)$ of the operator $T_2$ corresponding to the discs $\O_r(-2i)$ for all $r\le2$. 

Define 
$$
T_4\ =\ (U_3-2iI)\rho(|T_2+2iI|)+T_2(I-\rho(|T_2+2iI|))
$$ 
where $\rho$ is a nonincreasing $C^\infty$-function on $\R_+$ such that $\rho(t)=1$ for $t\le1$ and $\rho(t)=0$ for $t\ge2$. Since $U_3$ commutes with $\Pi_r(-2i)$ for $r\le2$ and coincides with the polar part $V_2$ of $T_2+2iI$ on the range of $\Pi_2(-2i)-\Pi_1(-2i)$, the operator $T_4$ has the following block structure 
\beq
\label{blocks}
T_4\ =\ (U_3-2iI)\Pi_1(-2i)\oplus \tilde T_2\bigl(\Pi_2(-2i)-\Pi_1(-2i)\bigr)\oplus T_2\bigl(I-\Pi_2(-2i)\bigr),
\eeq
where $\tilde T_2=T_2+(V_2-T_2-2iI)\rho(|T_2+2iI|)$. All operators in the right hand side of \eqref{blocks} are normal, and their spectra do not contain the point $-2i$. Thus $T_4$ is a normal operator whose spectrum is contained in 
$\Theta_2$ with a part of the lower arc removed, see Figure \ref{f3}. From the construction, it is clear that $T_4$ satisfies (i). The estimate (ii) for $T_4$ follows from the inclusion $\rho(|z|)\in C^2(\C)$ and the estimate (u$'_4$) for the operator $U_3$. If $\delta$ is small enough then the path $G_t=tT_4+(1-t)T_2-iI$ satisfy the assumptions of Lemma \ref{gl0}, which implies that $\diag_P T_4-iI\in \GL_0(\A\oplus \A)$.

\begin{figure}[H]
\begin{center}
\includegraphics[scale=0.7]{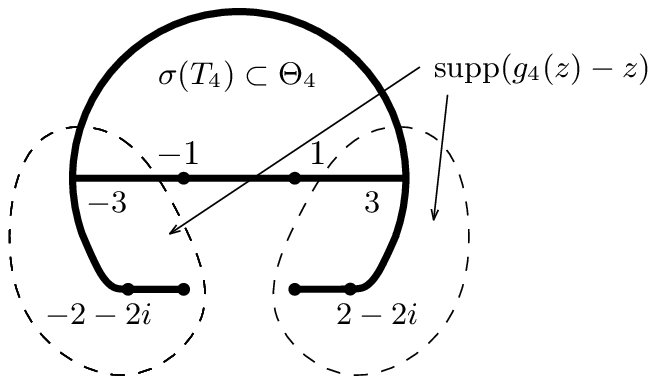}
\includegraphics[scale=0.7]{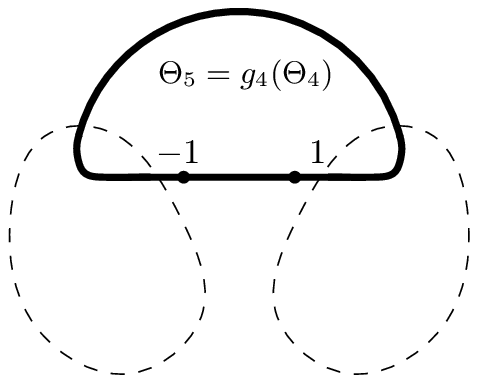}
\end{center}
\caption{The spectra of $T_4$ and $T_5=g_4(T_4)$}
\label{f3}
\end{figure}

Finally, there exists a smooth function $g_4\in\OL(\C)$ such that $g_4(z)=z$ outside the 
oval-shaped areas on Figure \ref{f3}, $g(z)-z\in C^2(\C)$ and  $g_4(\Theta_4)$ is given by an equation of the form \eqref{Gamma}. In simple words, $g_4$ maps the remaining parts of the lower arc into the end points of central line segment and does not affect the rest of $\Theta_4$, including the segment $[-2;2]$. The same arguments as before show that the operator $T_5=g_4(T_4)$ satisfies (i), (ii), and $\diag_P T_5-iI\in \GL_0(\A\oplus \A)$ provided that $\delta$ is sufficiently small. Since $\sigma(T_5)\subset g_4(\Theta_4)$, this completes the proof. \qed

\begin{remark}\label{r:all-operators}
If $\A$ is a von Neumann algebra then $\GL_0(\A\oplus\A)=\GL(\A)\oplus\GL(\A)$ (see Remark \ref{n-f}).
In this case we do not need to check the condition (iii) and can simplify Part 2 by taking $T_5=g(T)$, where $g\in C^2(\C)$ is an arbitrary function such that $g(\sigma(T))\subset\Theta_5$ and $g(z)=z$ near the line segment $[-2,2]$. 
\end{remark}

\begin{remark}\label{r:universal}
Note that the constants $C_j$ in the above proofs are determined only by the choice of the auxiliary functions $\rho$, $\rho_1$ and $g_j$. It follows that $C$ and $\delta$ in Theorems \ref{t:u1} and \ref{t:u2} are independent of $\A$ and $T$.
\end{remark}

\section{Approximation by operators with finite spectra}
\label{s:appr}

The main result of this section is  

\begin{theorem}
\label{main_approx}
There exist constants $\delta>0$ and $C>0$ such that the following is true. If $\A$ has real rank zero and $T\in \M2(\A)$ is a normal operator with $d_2(T)<\delta$ then one can find a normal operator $T_0\in \M2(\A)$ such that $\|T-T_0\|\le 3$, $\|[P,T_0]\|\le C d_2(T)$ and $\sigma(T_0)\subset \Z+ i\Z$.
\end{theorem}

\begin{remark}
\label{r:approx}
The above theorem can be thought of as an approximation result, since it holds with the same constants for operators $T$ with arbitrarily large norms. If $d_2(T)$ is sufficiently small then, applying Theorem \ref{main_approx} to $\eps^{-1}T$, we see that $T$ 
is approximated by an almost block diagonal operator $\eps T_0$ with finite spectrum $\sigma\bigl(\eps T_0\bigr)\subset \eps\Z+ i\eps\Z$.
\end{remark}

As was mentioned in Section \ref{s:main}, Theorem \ref{main_approx} is proved by removing certain subsets from the spectrum $\sigma(T)$. 
Theorems \ref{t:u1} and \ref{t:u2} from Section \ref{s:aux} allow us to remove discs from $\sigma(T)$ or, in other words, to cut holes in the spectrum. Note that, for our purposes, it is important to be able to cut arbitrarily many holes in one go. If we removed discs one by one repeatedly applying Theorems \ref{t:u1} and \ref{t:u2} as in  \cite{FS,FR,FR2}, then the error would accumulate and it would be difficult to effectively control the norm of the commutators with $P$. The following lemma deals with this issue, providing an estimate for $\|[P,T']\|$ independent of the number of holes.

\begin{lemma}
\label{manyholes}
There exist constants $\delta'>0$ and $C'>0$ such that the following is true for all $C^*$-algebras $\A$ of real rank zero and all normal $T\in \M2(\A)$. If $\lambda_{1},\ldots,\lambda_m\in \C$ is a finite collection of points such that $\dist(\lambda_i,\lambda_j)\ge 4$ for $i\neq j$ and each operator $T-\lambda_j I$ satisfies the conditions of Theorem {\rm\ref{t:u1}} or Theorem {\rm\ref{t:u2}} with $\delta<\delta'$ then one can find a normal $T'\in \M2(\A)$ with the following properties.
\begin{enumerate}
\item[{\rm(f$_1$)}]
$\|T-T'\|\le 3$.
\item[{\rm(f$_2$)}]
$\sigma(T')\cap\O_1(\lambda_j)=\varnothing$ for all $j=1,\ldots,m$.
\item[{\rm(f$_3$)}]
$\diag_P(T'-\lambda_j I)\in \GL_0(\A\oplus \A)$ for $j=1,\ldots,m$. 
\item[{\rm(f$_4$)}]
$[T',\Pi_{r}^j]=0$ and $\l.T'\r|_{\Ran \left(I-\sum_j\Pi_{2}^j\right)}=
\l.T\r|_{\Ran\left(I-\sum_j \Pi_{2}^j\right)}$  for all $r\in[1,2]$ and all $j=1,\ldots,m$, where $\Pi_r^j$ denote the spectral projections of $T$ corresponding to the discs $\O_{r}(\lambda_j)$.
\item[{\rm(f$_5$)}]
$\|[P,T']\|\le C'\delta$.
\item[{\rm(f$_6$)}]
If $\sigma(T)\cap\O_3(\lambda_j)$ lies on a straight line of the form $e^{i\theta}\R+\lambda_j$ then $\,\sigma(T')\cap\O_2(\lambda_j)\subset e^{i\theta}\R+\lambda_j$.
\end{enumerate}
\end{lemma}

\begin{proof}
Let $U_j$ be the unitary operators obtained by applying Theorem \ref{t:u1} or Theorem \ref{t:u2} to $T-\lambda_j I$. Let us fix a nonincreasing function
$\rho\in C^\infty(\mathbb R_+)$ such that 
$\rho(t)=1$ for $0\le t\le 1$ and $\rho(t)=0$ for $t\ge 2$, and denote $\chi_j(z)=\rho(|z-\lambda_j|)$. 
We claim that (f$_1$)--(f$_6$) hold for the operator 
\beq
\label{terdef}
T'=\sum\limits_{j=1}^m (\lambda_j I+U_j)\chi^2_j(T)+
T\Bigl\{I-\sum\limits_{j=1}^m \chi^2_j(T)\Bigr\}.
\eeq

Since the function $\chi_j(z)$ vanishes for $|z-\lambda_j|\ge 2$ and is equal to $1$ for $|z-\lambda_j|\le 1$, the condition (u$'_2$) implies that the operators in the first sum have block structure with respect to $\Pi_1^j$ and $\Pi_2^j$. So do the operators in the second sum, 
and hence the operator $T'$. From \eqref{terdef} and (u$'_3$) it follows that
\begin{enumerate}
\item[(1)]
the ``small'' blocks $\left.T'\right|_{\Ran\Pi_{1}^j}$ are the unitaries $\left.U_j\right|_{\Ran\Pi_{1}^j}$ shifted by $\lambda_j$,
\item[(2)]
the ``large'' block $\left.T'\right|_{\Ran\left(I-\sum_j \Pi_{2}^j\right)}$ coincides with the corresponding block of $T$,
\item[(3)]
the ``intermediate'' 
blocks $\left.T'\right|_{\Ran\l(\Pi_{2}^j-\Pi_{1}^j\r)}$ are functions of $\left.T\right|_{\Ran\l(\Pi_{2}^j-\Pi_{1}^j\r)}$ of the form 
$f_j(z)=\lambda_j+(z-\lambda_j)\bigl\{1+\bigl(|z-\lambda_j|^{-1}-1\bigr)\chi_j^2(z)\bigr\}\,$
with $\Ran f_j\subset\overline{\O_2(\lambda_j)\setminus\O_1(\lambda_j)}$.
\end{enumerate}
Thus we see that $T'$ is a normal operator satisfying (f$_2$) and (f$_4$). By (u$'_5$), we also have (f$_6$).

In view of (u$'_2$) and (u$'_3)$, $[U_j,\chi_j(T)]=0$ and, consequently, 
\begin{equation}
\label{nholes}
T'-T\ =\ \sum\limits_{j=1}^m \chi_j(T) U_j \chi_j(T)+
\sum\limits_{j=1}^m \chi_j(T) \l(\lambda_j I-T\r)\chi_j(T).
\end{equation}
Since $\chi_j$ have mutually disjoint supports, the norm of the first sum is bounded by 1. The second sum is a function of $T$ whose modulus does not exceed $2$. These estimates imply (f$_1$). 

In view of \eqref{intro1}, (u$_4$) and (u$'_4$), we have $\|[P,T]\|\le\delta$ and $\|[P,U_j]\|\le C\delta$ for all $j$ with some universal constant $C$. Therefore, the  inequality (f$_5$) follows from \eqref{nholes} and Lemma \ref{manysupp} with $Q=P$ and $S_j=\lambda_j I+U_j -T$. Note that the constant $C'$ in (f$_5$) depends only on $C$ and the choice of $\rho$.

It remains to prove (f$_3$). Let us fix $j$, denote $\mathcal T_j=\Pi_2^jT'+(I-\Pi_2^j)T$ and consider the path $G'_t=\Pi_2^jT'+(I-\Pi_2^j)\left((1-t)T+tT'\right)$ from $\mathcal T_j$ to $T'$. From (1)--(3) it follows that the operators $G'_t-\lambda_jI$ are invertible with $\|(G'_t-\lambda_jI)^{-1}\|\le1$. The inequality $\|[P,U_j]\|\le C\delta$ and \eqref{nholes} imply that 
$\|[P,\Pi_2^j(T-T')]\|\leq C_\rho\,\delta$, where $C_\rho$ is a constant depending only on $C$ and $\rho$. Hence $\|[P,G'_t]\|$ admits a similar estimate. If $\delta$ is smaller than a constant depending only on $C$ and 
$\rho$ then $\|[P,G'_t]]\|^{-1}>1$. Thus, in view of Lemma \ref{gl0}, it is sufficient to show that $\diag_P(\mathcal T_j-\lambda_j I)\in \GL_0(\A\oplus \A)$.

If $\,T-\lambda_jI\,$ satisfies the conditions of Theorem \ref{t:u1} then  $\,\diag_PU_j$ belongs to $\GL_0(\A\oplus \A)\,$ and, in view of (1)--(3) and (u$_3$), $\,U_j=(\mathcal T_j-\lambda_jI)|\mathcal T_j-\lambda_jI|^{-1}$. 
Choosing a homotopy $\ph_t(|\mathcal T_j-\lambda_jI|)$ from $|\mathcal T_j-\lambda_jI|^{-1}$ to $I$ with $\ph_t\in C^2(\R)$ and applying Lemma \ref{gl0}, we see that $\,\diag_P(\mathcal T_j-\lambda_j I)\in \GL_0(\A\oplus \A)\,$ provided that $\delta$ is smaller than a constant depending only on $C$ and 
the choice of $\ph_t$ and $\rho$.

Assume now that $\,T-\lambda_jI\,$ satisfies the conditions of Theorem \ref{t:u2}. Then the operators $\left.e^{-i\theta}U_j\right|_{\Ran\Pi_1^j}$ and $\left.e^{-i\theta}(T-\lambda_jI)\right|_{\Ran\Pi_2^j}$ are self-adjoint and $\,\diag_P(T-\lambda_jI+ie^{i\theta} I)\in \GL_0(\A\oplus \A)\,$. Consider the path
\beq
\label{path}
G_t:=t (\mathcal T_j-\lambda_jI)+(1-t)(T-\lambda_jI+ie^{i\theta} I)\,,\quad t\in[0,1].
\eeq 
Clearly, $\|[P,G_t]\|\le C'_\rho\delta$ with a constant $C'_\rho$ depending only on $C$ and $\rho$.

In view of (2), 
$\,\left.G_t\right|_{\Ran(I-\Pi_{2}^j)}=\left.(T-\lambda_jI+(1-t)ie^\theta I)\right|_{\Ran(I-\Pi_{2}^j)}\,$ 
and, consequently, the spectra of these restrictions are subsets of $\C\setminus\O_1$. Since the operator $e^{-i\theta}\left.(T-\lambda_jI)\right|_{\Ran\left(\Pi_2^j-\Pi_1^j\right)}$ is self-adjoint, (3) implies that the spectra of $\left.G_t\right|_{\Ran\left(\Pi_2^j-\Pi_1^j\right)}$ also lie outside $\O_1$. Finally, 
\begin{align*}
\re\l(e^{-i\theta}\left.G_t\right|_{\Ran\Pi_{1}^j}\r) &=\left.\left(tU_j+(1-t)(T-\lambda_j I)\right)\right|_{\Ran\Pi_{1}^j},\\
\im\l(e^{-i\theta}\left.G_t\right|_{\Ran\Pi_{1}^j}\r) &=(1-t)I
\end{align*}
because
the operators
$\left.e^{-i\theta}U_j\right|_{\Ran\Pi_1^j}$ and $\left.e^{-i\theta}(T-\lambda_jI)\right|_{\Ran\Pi_1^j}$ are self-adjoint.
By (u$'_5$), $\,\sigma\l(\left.e^{-i\theta}U_j\right|_{\Ran\Pi_1^j}\right)\subset\{-1;1\}$. The spectrum of $\left.e^{-i\theta}(T-\lambda_jI)\right|_{\Ran\Pi_1^j}$ is a subset of $[-1,1]$. These inclusions imply that $\sigma\left(\re(e^{-i\theta}\left.G_t\right|_{\Ran\Pi_{1}^j})\right)$
lie outside the interval $(-\frac13,\frac13)$ for all $t>\frac23$. It follows that  $\sigma\left(\left.G_t\right|_{\Ran\Pi_{1}^j}\right)$ are subsets of $\C\setminus\O_{1/3}$ for all $t\in[0,1]$.

Thus we see that $\|G_t^{-1}\|\le3$. Therefore, the path \eqref{path} satisfies 
the conditions of Lemma \ref{gl0} and, consequently, $\diag_P(\mathcal T_j-\lambda_j I)\in \GL_0(\A\oplus \A)$ provided that $\delta$ is smaller than a constant depending only on $C$ and $\rho$.
\end{proof}

\subsection*{Proof or Theorem \ref{main_approx}}

Let 
$$
\mathcal G:=(\R+i\,6\Z)\cup (6\Z+i\,\R)\quad\text{and}\quad\mathcal L:=(6\Z+3)+i(6\Z+3)\,.
$$
The grid $\mathcal G$ splits the complex plane into closed squares $\Omega_j$ of size $6\times6$ with vertices in $6\Z+i\,6\Z$. The lattice $\mathcal L$ is formed by the centre points $z_j$ of $\Omega_j$, and $\mathcal G=\cup_j\partial\Omega_j$.

Let us fix a nondecreasing function $\psi\in C^\infty(\R)$
such that the derivative $\psi'$ is periodic with period 6 and 
$\psi(t)=6k$ for all $t\in[6k-\frac52\,,\,6k+\frac52]$ and all $k\in\Z$. Define 
$$
g(z):=\psi(\re z)+i\,\psi(\im z)\quad\text{and}\quad g_t(z)=(1-t)z+tg(z)\,.
$$ 
Since $g(z)-z\in C^2(\C)$, the functions $g_t$ are operator Lipschitz.  One can easily see that $\,g\left(\Omega_j\setminus\O_1(z_j)\right)\subset\partial\Omega_j\,$ and
\beq
\label{g-t}
g_t\left(\Omega_j\setminus\O_1(z_j)\right)\subset \Omega_j\setminus \O_1(z_j)\,,\quad \forall t\in[0,1]\,.
\eeq

Put $T_1:=6T$ and consider a closed square
$\Omega$ with vertices in $6\Z+i\,6\Z$, which contains  $\sigma(T_1)$. Let us assume that $d_2(T_1)=6d_2(T)<\delta'$, apply
Lemma \ref{manyholes} with $\lambda_j=z_j\in\Omega\cap\mathcal L$ to the operator $T_1$, and denote the obtained normal operator by $T'_1$. The operators $T_1$ and $T'_1$ satisfy the conditions (f$_1$)--(f$_5$). In particular, $\|[P,T'_1]\|\le 6C'd_2(T)$, where $C'$ is the constant from (f$_5$), and $\sigma(T'_1)\subset\C\setminus\left(\cup_j\O_1(z_j)\right)$, see Figure \ref{gammae}.

\begin{figure}[H]
	\begin{center}
		\includegraphics[scale=0.5]{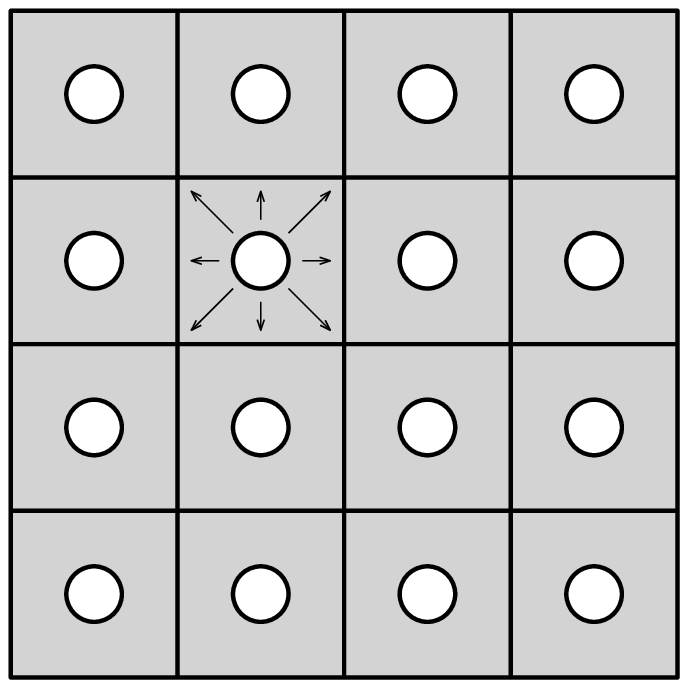}
		\hskip2cm
		\includegraphics[scale=0.5]{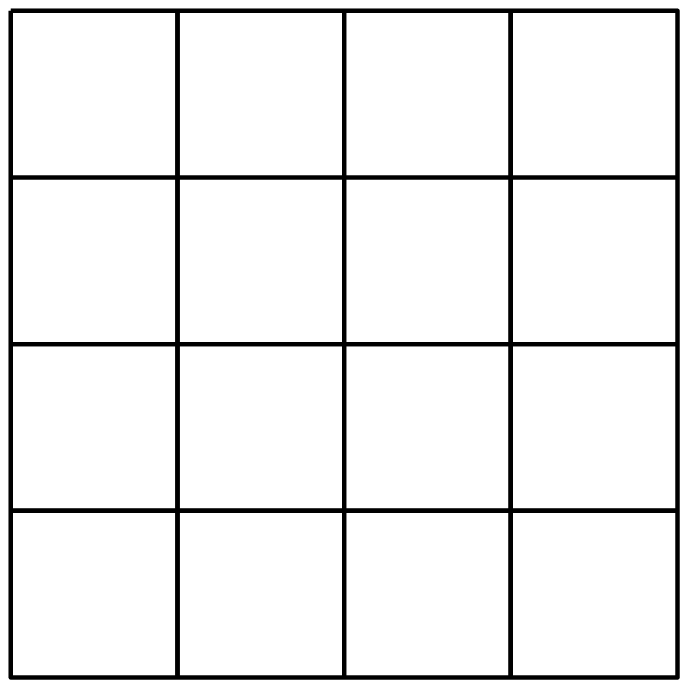}
	\end{center}
	\caption{The spectra of $T'_1$ and $T_2=g(T'_1)$}
	\label{gammae}
\end{figure}
Let $T_2=g(T'_1)$. Then $\sigma(T_2)\subset\mathcal G$ and, since $g$ is operator Lipschitz, $\|[P,T_2]\|\le C_\psi\,d_2(T)$ with a constant $C_\psi$ depending only on $C'$ and the choice of $\psi$. If $d_2(T)<C_\psi^{-1}$ then, in view of \eqref{g-t}, the paths $g_t(T'_1)-z_jI$ satisfy the conditions of Lemma \ref{gl0}. Hence $\diag_P(T_2-z_j I)\in \GL_0(\A\oplus \A)$ for all $z_j\in\Omega\cap\mathcal L$. Now, applying Lemma \ref{gl0} to the paths  $T_2-(tz +(1-t)z_j)I$, we see that $\diag_P(T_2-z I)\in \GL_0(\A\oplus \A)$ whenever $\dist(z,\mathcal G)\ge1$.

Let now $\lambda_j$ be the points of the union $(\mathcal L+3)\cup(\mathcal L+3i)$ lying in $\Omega$ (in other words, $\lambda_j$ are the middle points of the line segments forming the squares $\Omega_j$). If $d_2(T)$ is smaller than a constant depending only on the choice of $\psi$ then, by the above, the operators $T_2-\lambda_jI$ satisfy the conditions of Theorem \ref{t:u2}. Let us apply Lemma \ref{manyholes} to the operator $T_2$ and denote the obtained normal operator by $T'_2$. The spectrum $\sigma(T'_2)$ is also a subset of $\mathcal G$, but now it does not contain central parts of the edges of $\Omega_j$, see Figure \ref{f5}. By (f$_5$), we have $\|[P,T'_2]\|\le C'_\psi\,d_2(T)$ with a constant $C'_\psi$ depending only on $C'$ and $\psi$.

\begin{figure}[H]
\begin{center}
\includegraphics[scale=0.5]{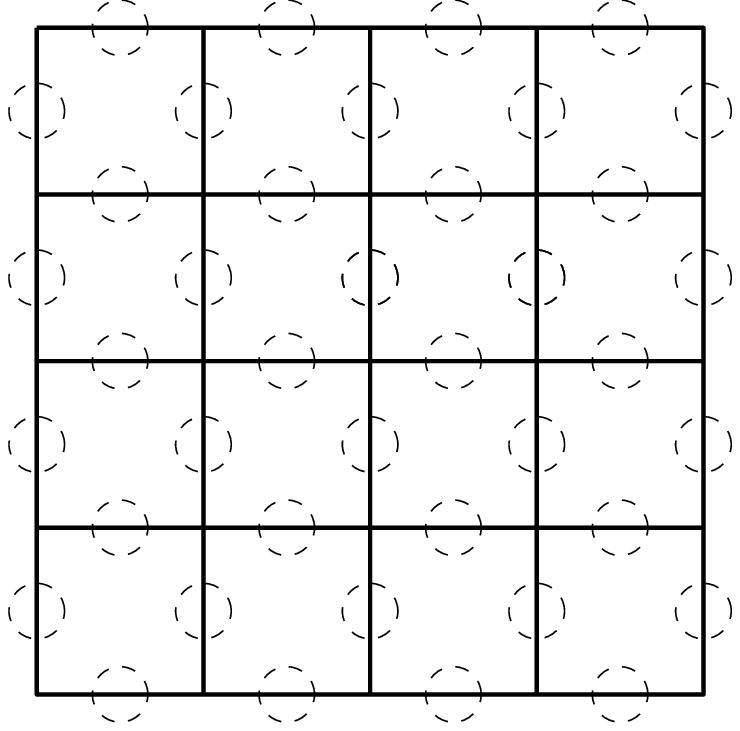}
\hskip2cm
\includegraphics[scale=0.5]{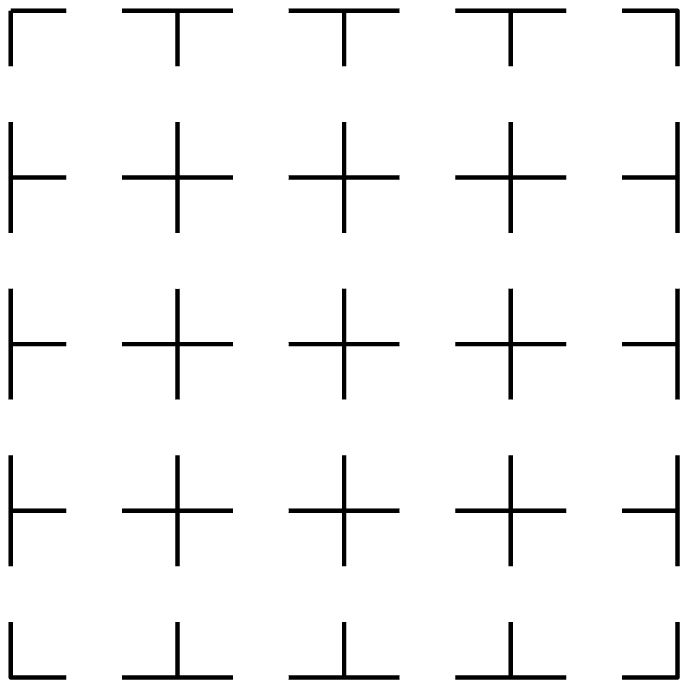}
\end{center}
\caption{The spectra of $T_2$ and $T'_2$}
\label{f5}
\end{figure}

The spectrum of $T'_2$ lies in the squares $\Omega'_k$ with edges of length 5 centred at the points $z'_k\in 6\Z+i\,6\Z$. On each of these squares the function $g$ is identically equal to $z'_k$. Thus the spectrum of $g(T'_2)$ is a finite subset of $6\Z+i\,6\Z$. 
Let $T_0=\frac16\,g(T'_2)$. Then $\sigma(T_0)\subset\Z+i\,\Z$. Since $g$ is operator Lipschitz, $\|[P,T_0]\|\leq C\,d_2(T)$ with a constant $C$ depending only on $C'$ in (f$_5$) and the choice $\psi$. Finally, the estimates (f$_1$) and $|g(z)-z|\le3\sqrt2$ imply that $\|T-T_0\|\le3$.
\qed

\section{Proofs of Theorem \ref{t:main} and Corollary \ref{c:bdf}}
\label{s:proofs}

\subsection{The upper bound \eqref{main-estimate1}}\label{s:upper-bound}

Denote $d'_1(A):=d_1(A)+\|[A,A^*]\|^{1/2}$.
Applying Theorem \ref{davidson} to the operator $A$, 
we find normal operators $N\in\A$ and $T\in \M2(\A)$ such that 
\beq
\label{estimates2}
d_2(T)\le C_1\,d'_1(A)
\quad\text{and}\quad
\|T-A\oplus N\|\le C_1\,d'_1(A).
\eeq
Here and further on, $C$ with a subscript denotes a constant independent of $\A$ and $A$.

Assume that $C_1\,d'_1(A)$ is smaller than the constant $\delta$ in Theorem \ref{main_approx}. Let $T_0$ be the normal operator with $\sigma(T_0)\subset\Z+i\,\Z$ given by that theorem, and let $PT_0P=X+iY$ where $X,Y$ are self-adjoint. We have $\|T-T_0\|\le 3$ and, in view of the first estimate \eqref{estimates2},
$\|T_0-\diag_PT_0\|=\|[P,T_0]\|\le C_2\,d'_1(A)$. These two inequalities,  the second estimate \eqref{estimates2} and the identity
$$
2i\,[X,Y]=[PT_0^*P,PT_0P]=PT_0^*(I-P)T_0P-PT_0(I-P)T_0^*P
$$ 
imply that
\beq
\label{estimates3}
\|X+iY-A\|\le 3+C_1\,d'_1(A) 
\quad\text{and}\quad
\|[X,Y]\|\leq C_3\left(d'_1(A)\right)^2.
\eeq
Assume also that $d'_1(A)\le(3C_2)^{-1}$, so that $\|T_0-\diag_PT_0\|\le \frac13$. Since $X$ is a block of $\re(\diag_P T_0)$ and  $\sigma(\re T_0)\subset \Z$, it follows that
\beq
\label{inclusion}
\sigma(X)\subset \sigma(\re(\diag_P T_0)) \subset \overbar\O_{1/3}(\Z)\cap\R\,,
\eeq
where $\overbar\O_{1/3}(\Z)$ is the closed $\frac13$-neighbourhood of $\Z$.

Let $\rho$ be a function satisfying the conditions of Lemma \ref{diagonalize} such that $\rho(t)=1$ for $|t|\le \frac13$ and $\rho(t)=0$ for $|t|\ge \frac23$.
If $\rho_n(t):=\rho(t-n)$ and  $Y'=\sum_{n\in\Z}\rho_n(X)Y\rho_n(X)$ then, by (d$_3$),
\beq
\label{estimates4}
\|Y-Y'\|\ \le\ C_\rho\,\|[X,Y]\|\ \le\ C_\rho\,C_3\left(d'_1(A)\right)^2.
\eeq

In view of \eqref{inclusion}, $\rho_n(X)$ coincides with the spectral projections of $X$ corresponding to the intervals $[n-\frac13,n+\frac13]$.
If $X':=\sum_{n\in\Z}n\,\rho_n(X)$ then $[X',Y']=0$. It follows that $X'+iY'$ is a normal operator whose spectrum lies on vertical line segments passing through real integers. 

Since $\A$ has real rank zero,  
$X'+iY'$ belongs to the closure of $\NN_f(\A)$. Now the estimates \eqref{estimates3}, \eqref{estimates4} and $\|X-X'\|\le\frac13$ imply that 
\beq
\label{estimates6}
\dist(A,\NN_f({\A}))\ \le\ 4+C_1\,d'_1(A)+C_\rho\,C_3\left(d'_1(A)\right)^2
\eeq
whenever $\,d'_1(A)\le\eps\,$, where $\eps$ is a constant depending only on the constants $\,\delta\,$ and $\,C\,$ in Theorems \ref{davidson} and  \ref{main_approx}. For every $A\in\A$, the operator $\eps\,(d'_1(A))^{-1}A$ satisfies the above condition. Substituting it into \eqref{estimates6}, we arrive at \eqref{main-estimate1}.

\subsection{The lower bound \eqref{main-estimate2}}\label{s:lower-bound}

Denote $d:=\dist(A,\NN_f({\A}))$, and let $N_\eps\in\NN_f({\A})$ be a normal operator with finite spectrum such that $\|A-N_\eps\|\le d+\eps$. If $R_\eps=A-N_\eps$ then $[A,A^*]=[R_\eps,A^*]+[A,R_\eps^*]-[R_\eps,R_\eps^*]$
and, consequently, $\|[A,A^*]\|\le4\|A\|\,(d+\eps)+(d+\eps)^2$. Since $d\le\|A\|$, letting $\eps\to0$, we obtain $\|[A,A^*]\|\le5\|A\|\,d$. Now \eqref{main-estimate2} follows from the inclusion $\NN_f({\A})\subset\overline{\GL_0(\A)}$ (see Remark \ref{gl-0}).

\subsection{Proof of Corollary \ref{c:bdf}}

Let $F$ be a continuous nondecreasing function on $\R_+$ such that $F(0)=0$ and $0\le F\le 1$. If 
$$
\dist\left(A,\NN({\A})\right)\ \le\  F\bigl(\|[A,A^*]\|\bigr)
$$
for any $C^*$-algebra $\A$ of real rank zero and all $A\in\A$ with $d_1(A)=0$ and $\|A\|\le1$ then, according to \cite[Theorem 3.8]{FS}, for every $\eps>0$ and
each $A\in\B(H)$ satisfying the conditions $d_1(A)=0$, $\|A\|\le1$, there
exists a normal operator $\,A_\eps\in\B(H)\,$ such that
\begin{equation}
\label{bdf1}
\begin{aligned}
&\|A-A_\eps\|_{\ess}\le 2F\bigl(\|[A^*,A]\|_{\ess}\bigr),\\
&\|A-A_\eps\|\le 5F\bigl(\|[A^*,A]\|\bigr)+
3F\Bigl(2F\bigl(\|[A^*,A]\|_{\ess}\bigr)\Bigr)+\eps\,.
\end{aligned}
\end{equation}
Since $\dist\left(A,\NN({\A})\right)\le\|A\|$, the estimate \eqref{main-estimate1} implies that the function $F(t)=\min\{Ct^{1/2},1\}$ satisfies the above conditions. Applying \eqref{bdf1} with this $F$ and $\eps=\|A\|^{-1}\|[A,A^*]\|^{1/2}$ to the operator $\|A\|^{-1}A$, we obtain \eqref{main-bdf}.

\begin{remark}
\label{bdf_rem}
If $\|(A-\lambda I)^{-1}\|<\left(\dist(\lambda,\Omega)-\eps\right)^{-1}$ whenever $\dist(\lambda,\Omega)>\eps$ and $\|[A^*,A]\|^{1/2}\le\eps$ then the spectrum of the normal operator $A'$ given by Corollary \ref{c:bdf} lies in a $C\er$-neighbourhood of $\Omega$. Therefore it can be approximated by a normal operator with spectrum in $\Omega$.
\end{remark}

\end{document}